\documentclass[11pt,reqno,british,twoside]{article}

\usepackage[utf8]{inputenc}
\usepackage[sc]{mathpazo}
\usepackage{babel,mathtools,amsmath,amssymb,url,paralist,xspace,mathrsfs}
\usepackage{booktabs,microtype,array,fixltx2e,fancyhdr}
\usepackage[margin=30pt,font=small,labelfont=bf,labelsep=period]{caption}
\usepackage[today,nofancy]{rcsinfo}
\usepackage[amsmath,thmmarks]{ntheorem}
\rmfamily
\rcsInfo $Id: skt4dsolv.tex,v 1.29 2009/11/03 10:15:14 swann Exp $

\mathtoolsset{mathic=true}
\setdefaultenum{\upshape (i)}{}{}{}

\theoremstyle{plain}
\theoremseparator{.}
\newtheorem{theorem}{Theorem}[section]
\newtheorem{proposition}[theorem]{Proposition}
\newtheorem{lemma}[theorem]{Lemma}
\newtheorem{corollary}[theorem]{Corollary}

\theorembodyfont{\normalfont}

\theoremheaderfont{\itshape}
\theoremsymbol{\begin{small}\ensuremath{\quad\triangle}\end{small}}
\newtheorem{remark}[theorem]{Remark}
\theoremsymbol{\begin{small}\ensuremath{\quad\diamondsuit}\end{small}}

\theoremstyle{nonumberplain}
\theoremheaderfont{\itshape}
\theorembodyfont{\normalfont}
\theoremsymbol{\begin{small}\ensuremath{\quad\Box}\end{small}}
\newtheorem{proof}{Proof}

\qedsymbol{\begin{small}\ensuremath{\quad\Box}\end{small}}

\numberwithin{equation}{section}
\numberwithin{table}{section}


\newcolumntype{C}{>{$}c<{$}}
\newcolumntype{L}{>{$}l<{$}}
\newcolumntype{R}{>{$}r<{$}}

\let\oldbibliography\thebibliography
\renewcommand{\thebibliography}[1]{%
  \oldbibliography{#1}%
  \small
  \setlength{\itemsep}{0pt}%
  \setlength{\parskip}{0pt}%
}

\newcommand{\lie}[1]{\mathfrak{#1}}
\newcommand{\g}{\lie{g}}
\newcommand{\h}{\lie{h}}

\newcommand{\n}{\lie{n}}
\newcommand{\lr}{\lie{r}}
\newcommand{\ld}{\lie{d}}

\newcommand{\lt}{\lie{t}}
\newcommand{\um}{\lie{u}}

\newcommand{\aff}{\lie{aff}}

\newcommand{\z}{\lie{z}}

\newcommand{\bC}{\mathbb C}
\newcommand{\bR}{\mathbb R}
\newcommand{\affR}{\aff_{\bR}}
\newcommand{\affC}{\aff_{\bC}}

\newcommand{\T}{\ensuremath{\checkmark}}
\newcommand{\X}{\ensuremath{\times}}

\DeclareMathOperator{\I}{\mathcal{I}}
\DeclareMathOperator{\ad}{ad}
\DeclareMathOperator{\im}{Im}
\DeclareMathOperator{\tr}{Tr}
\DeclareMathOperator{\rank}{rank}

\newcommand{\hook}{{\lrcorner}}

\newcommand{\HKT}{\textsc{hkt}\xspace}

\newcommand{\SKT}{\textsc{skt}\xspace}
\newcommand{\NB}{\nabla}
\newcommand{\LC}{\NB^{\textup{LC}}}
\newcommand{\Bis}{\textup{B}}
\newcommand{\cB}{c^{\Bis}}
\newcommand{\nB}{\NB^{\Bis}}
\newcommand{\TB}{T^{\Bis}}

\DeclarePairedDelimiter{\abs}{\lvert}{\rvert}

\DeclarePairedDelimiter{\Span}{\langle}{\rangle}

\begin{document}
\thispagestyle{empty}
\begin{small}
  \begin{flushright}
  IMADA-PP-2009-16\\
  CP\textsuperscript3-ORIGINS: 2009-22
  \end{flushright}
\end{small}

\bigskip

\begin{center}
  \LARGE\bfseries Invariant strong KT geometry on\\
  four-dimensional solvable Lie groups
\end{center}
\begin{center}
  \Large Thomas Bruun Madsen and Andrew Swann
\end{center}

\begin{abstract}
  A strong KT (SKT) manifold consists of a Hermitian structure whose
  torsion three-form is closed.  We classify the invariant SKT structures
  on four-dimensional solvable Lie groups.  The classification includes
  solutions on groups that do not admit compact four-dimensional quotients.
  It also shows that there are solvable groups in dimension four that admit
  invariant complex structures but have no invariant SKT structure.
\end{abstract}

\bigskip
\begin{center}
  \begin{minipage}{0.8\linewidth}
    \begin{small}
      \tableofcontents
    \end{small}
  \end{minipage}
\end{center}

\bigskip
\begin{small}\noindent
\end{small}
\newpage
\section{Introduction}
\label{sec:introduction}

On any Hermitian manifold \( (M,g,J) \) there is a unique Hermitian
connection~\cite{Gauduchon:Hermitian-Dirac}, called the Bismut connection,
which has torsion a three-form.  Explicitly the Bismut connection is given
by
\begin{equation}
  \nB = \LC+\tfrac12\TB,\qquad \cB = \bigl(\TB\bigr)^\flat = -Jd\omega,
\end{equation}
where \( \omega = g(J\cdot,\cdot) \) is the fundamental two-form and \(
Jd\omega = -d\omega(J\cdot,J\cdot,J\cdot) \).  If the torsion three-form~\(
\cB \) is closed, we have a \emph{strong Kähler manifold with torsion}, or
briefly an \emph{\SKT manifold}.  The study of \SKT structures has received
notable attention over recent years, see~\cite{Fino-T:skt} for a survey and
for an approach through generalized geometry,
see~\cite{Cavalcanti:metric-reduction}.  This has been motivated partly by
the quest for canonical choices of metric compatible with a given complex
structure and partly by the relevance of such geometries to super-symmetric
theories from physics
\cite{Gates-HR:twisted,Howe-P:further,Hull-LRvUZ:gKg,Michelson-S:conformal,Strominger:superstrings}.

Kähler manifolds are precisely the \SKT manifolds with torsion three-form
identically zero.  However, most \SKT manifolds are non-Kähler.  For
example compact semisimple Lie groups cannot be Kähler since they have
second Betti number equal to zero, but any even-dimensional compact Lie
group can be endowed with the structure of an \SKT manifold, see
Appendix~\ref{sec:compact}.  The \SKT geometry of nilpotent Lie groups was
studied by Fino, Parton \& Salamon \cite{Fino-PS:SKT}, who provided a full
classification in dimension~\( 6 \).

In this paper we classify \SKT structures on four-dimensional solvable Lie
groups, showing that there are a number of new examples; see
Table~\ref{tab:4solvskt}, only the first two entries belong to the
nilpotent classification.  The greater variety and complexity of this case
is already seen from the classification results for complex structures:
Salamon \cite{Salamon:complex-nil} classified the integrable complex
structures on \( 6 \)-dimensional nilpotent Lie groups, whereas in the
solvable case there is a classification only in dimension four
\cite{Andrada-BDO:four,Ovando:4,Snow:complex-solvable}.

In dimension four, a Hermitian manifold \( (M,g,J) \) is an \SKT manifold
precisely when the associated Lee one-form \( \theta = Jd^*\omega \) is
co-closed.  When \( M \) is compact, Gauduchon~\cite{Gauduchon:torsion}
showed that, up to homothety, there is a unique such metric in each
conformal class of Hermitian metrics.  The situation for non-compact
manifolds is less clear.  Our classification includes non-compact \SKT
manifolds that admit no compact quotient, and also shows that there are
invariant complex structures that admit no compatible invariant \SKT
metric.

\paragraph*{Acknowledgements}
We thank Martin Svensson for useful conversations and gratefully
acknowledge financial support from \textsc{ctqm} and \textsc{geomaps}.

\section{Solvable Lie algebras}
\label{sec:algebra}

Since we are interested in invariant structures on a simply-connected Lie
group~\( G \), it is sufficient to study the corresponding structures on
the Lie algebra \( \g \).  To \( \g \) one associates two series of ideals:
the \emph{lower central series}, which is given by \( \g_1 = \g' = [\g,\g]
\), \( \g_k = [\g,\g_{k-1}] \) and the \emph{derived series} defined by \(
\g^1 = \g' \), \( \g^k = [\g^{k-1},\g^{k-1}] \).  The Lie algebra is
\emph{nilpotent} (resp. \emph{solvable}) if its lower (resp. derived)
series terminates after finitely many steps.

One has that \( \g^j\subseteq\g_j \), so that nilpotent algebras are
solvable.  On the other hand, consider a solvable Lie algebra~\( \g \).
Lie's Theorem applied to the adjoint representation of the complexification
\( \g_{\bC} \), gives a complex basis for \( \g_{\bC} \) with respect to
which each \( \ad_X \) is upper triangular.  One then has:

\begin{lemma}
  \label{lem:nilsolv}
  A finite-dimensional Lie algebra \( \g \) is solvable if and only if its
  derived algebra \( \g' \) is nilpotent.  \qed
\end{lemma}

\begin{remark}
  \label{rem:4dsol}
  For \( \g \) solvable of dimension four, \( \g' \) has dimension at most
  three and so is one of a known list.  Lemma~\ref{lem:nilsolv} then
  implies that \( \g' \) is either Abelian or the Heisenberg algebra \(
  \h_3 \), which has basis elements \( E_1,E_2,E_3 \) with only one
  non-trivial Lie bracket \( [E_1,E_2] = E_3 \).
\end{remark}

Identifying \( \g \) with left-invariant vector fields on \( G \), and \(
\g^* \) with left-invariant one-forms one has the relation
\begin{equation*}
  da(X,Y) = -a([X,Y])
\end{equation*}
for all \( X,Y\in\g \) and \( a\in\g^* \).  We may describe for example \(
\h_3 \) by letting \( e_1,e_2,e_3 \) be the dual basis in \( \g^* \) to \(
E_1,E_2,E_3 \) and computing \( de_1 = 0 \), \( de_2 = 0 \), \( de_3 =
e_2\wedge e_1 \).  We will use the compact notation \( \h_3 = (0,0,21) \)
to encode these relations.

Let \( \Lambda^*\g^* \) be the exterior algebra on~\( \g^* \) and write \(
\I(A) \) for the ideal in \( \Lambda^*\g^* \) generated by a subset \( A
\).  We interpret the condition for \( \g \) to be solvable dually via:

\begin{lemma}
  \label{lem:solvbase}
  A finite-dimensional Lie algebra \( \g \) is solvable if and only if
  there are maximal subspaces \( \{0\} = W_0<W_1<\dots<W_r = \g^* \) such
  that
  \begin{equation}
    \label{eq:W}
    dW_i \subseteq \I(W_{i-1})
  \end{equation}
  for each \( i \).  \qed
\end{lemma}

Concretely \( W_1 = \ker (d\colon \g^*\to\Lambda^2\g^*) \) and \( W_i \) is
defined inductively to be the maximal subspace satisfying~\eqref{eq:W}.  We
will sometimes find it useful to choose a filtration \( \{0\} =
V_0<V_1<\dots <V_n = \g^* \) with
\begin{equation}
  \label{eq:rbase} 
  \dim_{\bR} V_i = i\quad\text{and}\quad dV_i\subseteq
  \I(V_{i-1})\qquad\text{for each \( i \)}.
\end{equation}
One way to construct such filtrations is to refine the spaces \( W_i \),
however in some cases other choices may be possible and useful.

\section{The SKT structural equations}
\label{sec:SKT}

A left-invariant almost Hermitian structure on \( G \) is determined by an
inner product \( g \) on the Lie algebra~\( \g \) and a linear endomorphism
\( J \) of~\( \g \) such that \( J^2 = -1 \) and \( g(JX,JY) = g(X,Y) \)
for all \( X,Y\in\g \).  The \SKT condition consists of the requirement
that \( J \) be integrable and that \( dJd\omega = 0 \) where \(
\omega(X,Y) = g(JX,Y) \).  In the differential algebra, integrability of~\(
J \) may be expressed as the condition that \( d\Lambda^{1,0}\subseteq
\Lambda^{2,0}+\Lambda^{1,1} \).  If \( \g \) is four-dimensional and
solvable, we now show that there is one of two choices of possible good
bases \( \{a,Ja,b,Jb\} \) for~\( \g^* \).  We will later determine the \SKT
condition in each case.

\begin{lemma}
  \label{lem:struct}
  Let \( \g \) be a solvable Lie algebra of dimension four.  If \( (g,J) \)
  is an integrable Hermitian structure on~\( \g \) then there is an
  orthonormal set \( \{a,b\} \) in \( \g^* \) such that \( \{a,Ja,b,Jb\} \)
  is a basis for \( \g^* \) and either
  \begin{asparadesc}
  \item[Complex case:] \( \g \) has structural equations
    \begin{equation}
      \label{eq:strC}
      \begin{gathered}
        da = 0,\quad d(Ja) = x_1aJa,\quad
        db = y_1aJa+y_2ab+y_3aJb+z_1bJa+z_2JaJb,\\
        d(Jb) = u_1aJa+u_2ab+u_3aJb+v_1bJa+v_2JaJb+w_1bJb,
      \end{gathered}
    \end{equation} or
  \item[Real case:] \( \g \) has structural equations
    \begin{equation}
      \label{eq:strR}
      \begin{gathered}
        da = 0,\quad
        d(Ja) = x_1aJa+x_2(ab+JaJb)+x_3(aJb+bJa)+y_2bJb,\\
        db = z_1aJa+z_2ab+z_3aJb,\\
        d(Jb) = u_1aJa+u_2ab+u_3aJb+v_1bJa+v_2bJb+w_1JaJb.
      \end{gathered}
    \end{equation}
  \end{asparadesc}
  In the complex case, \( \{a,Ja,b,Jb\} \) may be chosen orthonormal and \(
  \omega = aJa+bJb \), omitting \( \wedge \) signs.  In the real case, \(
  \omega = aJa+bJb + t(ab+JaJb) \) for some \( t\in (-1,1) \).
\end{lemma}

\begin{proof}
  Let \( V_i \) be a refined filtration of \( \g \) as in~\eqref{eq:rbase}.
  As \( \dim_{\bR}V_2 = 2 \) we have two possibilities for the complex
  subspace \( V_2\cap JV_2 \), either it is non-trivial so \( V_2 = JV_2 \)
  or it is zero.  If the filtration \( V_i \) can be chosen with \( V_2 =
  JV_2 \) we will say we are in the complex case, otherwise we are in the
  real case.

  For the complex case, \( JV_2 = V_2 \) and \( V_1\subseteq V_2\cap\ker d
  \), so we may take an orthonormal basis \( \{a,Ja\} \) of \( V_2 \) with
  \( a\in V_1 \).  We have \( da = 0 \) and solvability implies \( d(Ja)
  \in \I(a) \cap \Lambda^2 = \bR aJa \oplus a\wedge V_2^\bot \).  As \( J
  \) is integrable, we must have \( d(Ja) \in \Lambda^{1,1} \) too, so \(
  d(Ja) = x_1aJa \).

  In the real case, choose \( a\in V_1 \) and \( b\in V_2\cap V_1^\bot \)
  of unit length.  Then \( da = 0 \) and the form of \( d(Ja) \) follows
  from the condition \( d(Ja) \in \Lambda^{1,1} \).  The form of \( \omega
  \) follows from \( t = g(b,Ja) \) which has absolute value less than \( 1
  \) by the Cauchy-Schwarz inequality.
\end{proof}

The above equations are necessary but far from sufficient.  For
integrability it remains to impose \( d(b-iJb)^{0,2} = 0 \), and to obtain
a Lie algebra we must satisfy the Jacobi identity.  The latter is
equivalent to the condition \( d^2 = 0 \).  Both of these conditions are
straightforward to compute.  We list the results below.  In each case the
first line comes from the integrability condition on~\( J \), in the last
line we provide the \SKT condition and the remaining equations are from \(
d^2 = 0 \).

\begin{lemma}
  The structural equations of Lemma~\ref{lem:struct} give an \SKT structure
  on a solvable Lie algebra if and only if the following quantities vanish:
  \begin{asparadesc}
  \item[Complex case:]
    \begin{equation}
      \label{eq:condC}
      \begin{gathered}
        y_2-z_2-u_3+v_1,\quad y_3-z_1+u_2-v_2,\\
        x_1z_1-y_3v_1-z_2u_2,\quad (x_1-y_2+u_3)z_2-y_3(z_1+v_2),\\
        y_2w_1,\quad y_3w_1,\quad z_1w_1,\quad z_2w_1,\\
        (x_1+y_2-u_3)v_1-(z_1+v_2)u_2+u_1w_1,\\
        x_1v_2+y_1w_1-y_3v_1-z_2u_2,\\
        (x_1+y_2+u_3)(y_2+u_3)+(z_1-v_2)^2-u_1w_1.
      \end{gathered}
    \end{equation}
  \item[Real case:]
    \begin{equation}
      \label{eq:condR}
      \begin{gathered}
        z_2-u_3+v_1,\quad z_3+u_2-w_1,\\
        x_2u_2-x_3(z_2-v_1)-y_2u_1,\quad (-x_1+z_2+u_3)y_2+x_2^2+x_3(x_3-v_2),\\
        x_2u_3-x_3(w_1+z_3)+y_2z_1,\quad (x_1+z_2-u_3)v_1-(x_3-v_2)u_1-u_2w_1,\\
        x_2v_2-y_2w_1,\quad x_3z_1+z_3v_1,\quad y_2z_1+z_3v_2,\quad
        x_2z_1+z_3w_1,\quad x_2v_1-x_3w_1,\\
        x_2w_1+x_3v_1-y_2u_1+z_2v_2,\quad x_1w_1-x_2u_1+z_1v_2-z_3v_1,\\
        \begin{multlined}
          \bigl\{(x_1+z_2+u_3)(-y_2+z_2+u_3)+x_2(x_2-z_1+tv_2)\qquad\\
          +\bigl(x_3-u_1+t(u_2-w_1)\bigr)(x_3+v_2) +w_1^2\bigr\}.
        \end{multlined}
      \end{gathered}
    \end{equation}
  \end{asparadesc}
\end{lemma}

In some cases the \SKT structure reduces to Kähler.  This occurs if and only if
the following additional conditions hold:
\begin{asparadesc}
\item[Complex case:]
  \begin{equation}
    y_1 = 0 = u_1,\quad u_3 = -y_2,\quad v_2 = z_1\label{eq:KcondC}
  \end{equation}
\item[Real case:]
  \begin{equation}
    \label{eq:KcondR}
    \begin{gathered}
      x_2-z_1 = t(x_1+u_3),\quad x_3 - u_1 = - tu_2,\quad y_2 - z_2 - u_3 =
      tx_2,\\
      w_1 = t(x_3 + v_2).
    \end{gathered}
  \end{equation}
\end{asparadesc}

\section{Low-dimensional solvable Lie algebras}

The four-dimensional solvable real Lie algebras are classified in
\cite{Andrada-BDO:four}, and we shall identify the algebras we obtain with
algebras on the known list. In this section we summarise the classification
and notation.

The map \( \chi\colon\g\to\bR \), \( \chi(x) = \tr(\ad(x)) \), is a Lie
algebra homomorphism.  Its kernel~\( \um(\g) \), the \emph{unimodular
kernel of \( \g \)}, is an ideal in \( \g \) containing the derived algebra
\( \g' \).  The Lie algebra \( \g \) is said to be \emph{unimodular} if \(
\chi\equiv0 \).  Note that if \( G \) admits a co-compact discrete subgroup
then \( \g \) is necessarily unimodular~\cite{Milnor:left}.

Our notation for the three-dimensional solvable Lie algebras will be as
given in Table \ref{tab:3solv}.  Note that \( \lr_{3,0}\cong
\bR\times\aff_\bR \).
\begin{table}[htp]
  \centering
  \begin{tabular}{LLL}
    \toprule
    \affR            & (0,21)                           &                         \\
    \h_3             & (0,0,21)                         &                         \\
    \lr_3            & (0,21+31,31)                     &                         \\
    \lr_{3,\lambda}  & (0,21,\lambda31)                 & \abs{\lambda}\leqslant1 \\
    \lr'_{3,\lambda} & (0,\lambda 21+31,-21+\lambda 31) & \lambda\geqslant0       \\
    \bottomrule
  \end{tabular}
  \caption{Non-Abelian solvable Lie algebras of dimension at most three
  that are not of product type.}
  \label{tab:3solv}
\end{table}

The four dimensional solvable Lie algebras are classified as follows.

\begin{theorem}[\cite{Andrada-BDO:four}]
  \label{theorem:clsfsol4}
  Let \( \g \) be a four dimensional solvable real Lie algebra. Then \( \g
  \) is isomorphic to one and only one of the following Lie algebras: \(
  \bR^4 \), \( \affR\times\affR \), \( \bR\times\h_3 \), \( \bR\times\lr_3
  \), \( \bR\times\lr_{3,\lambda} \) \( (\abs{\lambda}\leqslant 1) \), \(
  \bR\times\lr'_{3,\lambda} \) \( (\lambda\geqslant 0) \), or one of the
  algebras in Table~\ref{tab:4dnp}.

  Among these the unimodular algebras are: \( \bR^4 \), \( \bR\times\h_3
  \), \( \bR\times\lr_{3,-1} \), \( \bR\times\lr'_{3,0} \), \(
  \mathfrak{n}_4 \), \( \lr_{4,-1/2} \), \( \lr_{4,\mu,-1-\mu} \) \(
  (-1<\mu\leqslant-\frac12) \), \( \lr'_{4,\mu,-\mu/2} \), \( \ld_4 \), \(
  \ld'_{4,0} \).
\end{theorem}

\begin{table}[htp]
  \centering
  \begin{tabular}{LLL}
    \toprule
    \n_4&(0,0,21,31)&\\
    \affC&(0,0,31-42,41+32)&\\
    \lr_4&(0,21+31,31+41,41)&\\
    \lr_{4,\lambda}&(0,21,\lambda31+41,\lambda41)&\\
    \lr_{4,\mu,\lambda}&(0,21,\mu31,\lambda41)&\mu,\lambda\in \mathscr R_4\\
    \lr'_{4,\mu,\lambda}&(0,\mu21,\lambda31+41,-31+\lambda41)&\mu>0\\
    \ld_4&(0,21,-31,32)\\
    \ld_{4,\lambda}&(0,\lambda21,(1-\lambda)31,41+32)&\lambda\geqslant\tfrac12\\
    \ld'_{4,\lambda}&(0,\lambda21+31,-21+\lambda31,2\lambda.41+32)&\lambda\geqslant 0\\
    \h_4&(0,21+31,31,2.41+32)&\\
    \bottomrule
  \end{tabular}
  \caption{Four-dimensional solvable Lie algebras not of product type.  The
  set $\mathscr R_4$ consists of the $(\mu,\lambda)\in [-1,1]^2$ with
  $\lambda\geqslant\mu$ and $\mu,\lambda\ne0$ and satisfying
  $\lambda<0$ if $\mu = -1$.}
  \label{tab:4dnp}
\end{table}

In the Table \ref{tab:4solv} the four-dimensional solvable real Lie
algebras are sorted by their derived algebra \( \g' \).
\begin{table}[htp]
  \centering
  \begin{tabular}{LLL}
    \toprule
    \g' &\z(\g)& \g\\
    \midrule
    \{0\} && \bR^4\\
    \bR && \bR\times\h_3,\  \bR\times\lr_{3,0}\\
    \midrule
    \bR^2&\{0\}& \affR\times\affR,\  \affC,\  \ld_{4,1}\\
    &\bR& \bR\times\lr_3,\  \bR\times\lr_{3,\lambda\ne0},\
    \bR\times\lr'_{3,\lambda},\  \lr_{4,0},\  \n_{4}\\ \midrule 
    \bR^3 && \lr_4,\ \lr_{4,\lambda\ne0},\
    \lr_{4,\mu,\lambda},\ \lr'_{4,\mu,\lambda}\\ 
    \h_3 && \ld_4,\  \ld_{4,\lambda\ne1},\  \ld'_{4,\lambda},\ \h_4\\   
    \bottomrule
  \end{tabular}
  \caption{The four-dimensional solvable Lie algebras sorted by $\g'$ and,
  where necessary, $\z(\g)$. The conditions on the parameters are in
  addition to those from Tables~\ref{tab:3solv} and~\ref{tab:4dnp}.}
  \label{tab:4solv}	
\end{table}
In some cases it is easy to recognise which algebra is at hand using the
following observations:
\begin{asparadesc}
\item[\( \g' = \bR \):] \( \bR\times\h_3 \) is nilpotent, \(
  \bR\times\lr_{3,0} \) is not.
\item[\( \g' = \bR^2,\ \z(\g) = \{0\} \):] \( \affR\times\affR \) and \(
  \ld_{4,1} \) are completely solvable, \( \affC \) is not. Moreover these
  algebras have different unimodular kernels:
  \begin{equation*}
    \um(\affR\times\affR)\cong \lr_{3,-1},\quad \um(\ld_{4,1})\cong
    \h_3,\quad \um(\affC)\cong \lr'_{3,0}.
  \end{equation*}
\item[\( \g' = \h_3 \):] the algebras are distinguished by \( \tilde\g =
  \g/\z(\g') \) as follows:
  \begin{equation*}
    \tilde \ld_4 \cong \lr_{3,-1}       ,\quad 
    \tilde \ld_{4,\lambda\ne1} \cong \lr_{3,(1-\lambda)/\lambda} ,\quad
    \tilde \ld'_{4,\lambda} \cong  \lr'_{3,\lambda} ,\quad 
    \tilde \h_4  \cong \lr_3.
  \end{equation*}
\end{asparadesc}

\section{The SKT classification}
\label{sec:clsfskt}
We are now ready to describe the simply-connected four-dimensional solvable
real Lie groups admitting invariant \SKT structures.

\begin{theorem}
  \label{thm:clsfskt1}
  Let \( G \) be a simply-connected four-dimensional solvable real Lie
  group.  Then \( G \) admits a left-invariant \SKT structure if and only if
  its Lie algebra \( \g \) is listed in Table~\ref{tab:4solvskt}.
\end{theorem}

The table also indicates which groups admit invariant Kähler metrics, and
gives the dimensions of the Lie algebra cohomology.

\begin{table}[htp]
  \centering
  \begin{tabular}{LLCCcC}
    \toprule
    \g'   & \g                                      & \dim & \pi_0 & Kähler & (b_1\dots b_4) \\
    \midrule
    \{0\} & \bR^4                                   & 0    & 1     & \T     & (4,6,4,1)      \\
    \midrule
    \bR   & \bR\times\h_3                           & 0    & 1     & \X     & (3,4,3,1)      \\
          & \bR\times\lr_{3,0}                      & 1    & 1     & \T     & (3,3,1,0)      \\
    \midrule
    \bR^2 & \bR\times\lr'_{3,0}                     & 1    & 1     & \T     & (2,2,2,1)      \\ 
          & \affR\times\affR                        & 2    & 1     & \T     & (2,1,0,0)      \\			
    \midrule
    \bR^3 & \lr'_{4,\lambda,0}\  (\lambda>0)        & 1    & 2     & \T     & (1,1,1,0)      \\
          & \lr_{4,-1/2,-1/2}                       & 1    & 1     & \X     & (1,0,1,1)      \\
          & \lr'_{4,2\lambda,-\lambda}\ (\lambda>0) & 1    & 2     & \X     & (1,0,1,1)      \\
    \midrule
    \h_3  & \ld_{4}                                 & 2    & 1     & \X     & (1,0,1,1)      \\
          & \ld_{4,2}                               & 2    & 1     & \T     & (1,1,1,0)      \\
          & \ld'_{4,0}                              & 2    & 1     & \X     & (1,0,1,1)      \\
          & \ld_{4,1/2}                             & 1    & 1     & \T     & (1,0,0,0)      \\	
          & \ld'_{4,\lambda}\ (\lambda>0)           & 1    & 1     & \T     & (1,0,0,0)      \\
    \bottomrule
  \end{tabular}
  \caption{The four-dimensional solvable Lie algebras that admit an \SKT
  structure.  Of these, only $\bR^4$ fails to admit an \SKT structure
  that is not Kähler.  In the table, $\dim$ and $\pi_0$ are the dimension and number of
  components of the \SKT moduli space modulo homotheties, $b_k$ denotes
  $\dim H^{k}(\g)$.}
  \label{tab:4solvskt}
\end{table}

The proof will occupy the rest of this section.  Following
Remark~\ref{rem:4dsol} we analyse the possible solutions to the equations
of~\S\ref{sec:SKT} case by case after the type of~\( \g' \).  When talking
of the \SKT moduli space, we consider only left-invariant structures on the
given \( \g \) and regard two structures as equivalent if one may be
obtained from the other via a Lie algebra automorphism of~\( \g \).

\subsection{Trivial derived algebra}

For \( \g' = \{0\} \), \( \g\cong\bR^4 \) is Abelian, \( d\equiv0 \) so all
structure constants are zero and each almost Hermitian structure is Kähler.
All these Kähler structures are equivalent.

\subsection{One-dimensional derived algebra}
\label{sec:1d}

For \( \g' = \bR \), we have \( \dim W_1 = 3 \). It follows that we can
choose \( a,\,Ja,\,b \in W_1 \) and are thus in the case \( V_2 = JV_2
\). The structural equations for \( \g \) in this case are
\begin{gather*}
  da = 0 = d(Ja) = db,\\
  d(Jb) = u_1aJa+u_2(ab+JaJb)+u_3(aJb+bJa)+w_1bJb,
\end{gather*}
where the coefficients satisfy \( 0 = u_2^2+u_3^2-u_1w_1 \) and \(
d(Jb)\ne0 \).  Rotating \( a,Ja \) in~\( V_2 \), we may ensure that \( u_2
= 0 \) and \( u_3\geqslant 0 \), so \( u_1w_1 = u_3^2 \).  Replacing \( b
\) by \( -b \), we obtain \( w_1\geqslant0 \).

If \( w_1 = 0 \) then \( u_3 = 0 \), \( u_1 \ne 0 \). Thus we have the
algebra given by
\begin{gather}
  da = 0 = d(Ja) = db,\quad d(Jb) = u_1aJa.
\end{gather}
The resulting \SKT metrics admitted are non-Kähler and homothetic.
Moreover we see that \( \g \) is nilpotent and so isomorphic to \(
\bR\times\h_3 \).

If \( w_1> 0 \) then \( \g \) is not nilpotent and so isomorphic to \(
\bR\times\lr_{3,0} \).  As \( u_1w_1 = u_3^2 \geqslant 0 \) we have the
structural equations
\begin{gather*}
  da = 0 = d(Ja) = db,\quad d(Jb) = u_1aJa+u_3(aJb+bJa)+w_1bJb,
\end{gather*}
with \( u_3 = \sqrt{u_1w_1} \), \( u_1\geqslant0 \).  This is Kähler only
if \( u_1 = 0 \).  Up to homothety the only parameter is~\( u_1 \).  The
moduli space is thus connected.

\subsection{Two-dimensional derived algebra}
\label{sec:2d}

For \( \g' = \bR^2 \), we have \( \dim W_1 = 2 \), and we shall distinguish
between the cases \( W_1 = JW_1 \) and \( W_1 \cap JW_1 = \{0\} \) where \(
W_1 = \ker d \) is complex or real.

\subsubsection{Complex kernel}
\label{sec:2dC}

We have \( W_1 = JW_1 \) and taking \( V_2 = W_1 \) thus have the
structural equations
\begin{gather*}
  da = 0 = d(Ja),\\
  db = y_1aJa+y_3aJb+z_2JaJb,\\
  d(Jb) = u_1aJa-y_3ab+z_2bJa
\end{gather*}
with no restrictions on the coefficients other than that \( db \) and \(
d(Jb) \) are linearly independent.  Rotating \( a,Ja \) we may put \( z_2 =
0 \), \( y_3>0 \).  Rotating \( b,Jb \) we can then get \( u_1\geqslant 0
\), \( y_1 = 0 \), reducing the structure to
\begin{equation*}
  da = 0 = d(Ja),\quad
  db = y_3aJb,\quad
  d(Jb) = u_1aJa-y_3ab.
\end{equation*}
The solution is Kähler if and only if \( u_1 = 0 \).  The \SKT moduli space
is connected of dimension \( 1 \) modulo homotheties.  The Lie algebra \(
\g \) is isomorphic to \( \bR\times\lr'_{3,0} \).

\subsubsection{Real kernel}
\label{sec:2dR}

Here \( W_1 \cap JW_1 = \{0\} \) and we again take \( V_2 = W_1 \) putting
us in the real case and giving the structural equations
\begin{gather*}
  da = 0 = db,\\
  d(Ja) = x_1aJa+x_3(aJb+bJa)+y_2bJb,\\
  d(Jb) = u_1aJa+u_3(aJb+bJa)+v_2bJb,
\end{gather*}
where the last two lines are linearly independent and the coefficients
satisfy
\begin{equation}
  \label{eq:2dR}
  \begin{gathered}
    (x_1-u_3)y_2 = (-v_2+x_3)x_3,\quad u_1(v_2-x_3) = u_3(u_3-x_1),\\
    u_3x_3 = u_1y_2,\quad (u_1-x_3)(v_2+x_3) = (u_3+x_1)(u_3-y_2).
  \end{gathered}
\end{equation}

\begin{lemma}
  We have \( \z(\g) = \{0\} \) and \( \um(\g) \cong \lr_{3,-1} \), so \(
  \g\cong \affR\times\affR \).
\end{lemma}

\begin{proof}
  We compute the centre via \( \z(\g) = \{\,X\in\g:X\hook d\alpha = 0\text{
  for all \( \alpha\in\g^* \)}\,\} \).  Writing \( X = pA +qB + p'JA + q'JB
  \), where \( \{A,B,JA,JB\} \) is the dual basis to \( \{a,b,Ja,Jb\} \),
  one finds that \( X\in\z(\g) \) implies \( (p,q,0)^T \) and \( (0,p,q)^T
  \) lie in the one-dimensional null space of the rank two matrix
  \begin{equation*}
    Q =
    \begin{pmatrix}
      x_1&x_3&y_2\\
      u_1&u_3&v_2
    \end{pmatrix}
    .
  \end{equation*}
  We conclude that \( p = 0 = q \).  The same calculation applies to \( p'
  \) and \( q' \), so \( X = 0 \) and \( \z(\g) = \{0\} \).

  Writing \( \mathbf a = \left(
    \begin{smallmatrix}
      x_1\\x_3
    \end{smallmatrix}
  \right) \), \( \mathbf b = \left(
    \begin{smallmatrix}
      x_3\\y_2
    \end{smallmatrix}
  \right) \), \( \mathbf c = \left(
    \begin{smallmatrix}
      u_1\\u_3
    \end{smallmatrix}
  \right) \), \( \mathbf d = \left(
    \begin{smallmatrix}
      u_3\\v_2
    \end{smallmatrix}
  \right) \), equations~\eqref{eq:2dR} may be interpreted geometrically as
  saying that \( \mathbf b \), \( \mathbf c \) and \( \mathbf a-\mathbf d
  \) are mutually parallel and that \( \mathbf b-\mathbf c \) is parallel
  to \( \mathbf a + \mathbf d \).  Imposing the constraint \( \rank Q = 2
  \), then leads to the fact that \( \mathbf a \) and \( \mathbf d \) are
  linearly independent.

  The map \( \chi = \tr\ad\colon \g\to\bR \) is given by \( \chi(A) =
  -(x_1+u_3) \), \( \chi(B) = -(x_3+v_2) \), \( \chi(JA) = 0 = \chi(JB) \).
  This is zero only if \( \mathbf a = -\mathbf d \), which by the above
  remark, is not possible.  Thus \( \g \) is not unimodular.  Choosing \(
  a\in \im\chi^* \leqslant\ker d \), we have \( 0 = a(B) \varpropto \chi(B)
  \) and so \( v_2 = -x_3 \).

  Write \( \mathbf a-\mathbf d = 2k\mathbf v \) with \( \mathbf v = \left(
    \begin{smallmatrix}
      c\\s
    \end{smallmatrix}
  \right) \), \( c^2+s^2 = 1 \).  Then \eqref{eq:2dR} implies \( \mathbf
  b,\mathbf c\in \Span{\mathbf v} \).  However \( \mathbf a+\mathbf
  d\notin\Span{\mathbf v} \) but is parallel to \( \mathbf b-\mathbf c \),
  so we find \( \mathbf b = \mathbf c = h\mathbf v \), for some \( h\in\bR
  \).  This gives \( x_3 = ks = hc \), so we may write \( k = \ell c \), \(
  h = \ell s \) for some non-zero \( \ell\in\bR \).  Changing the sign of
  \( \mathbf v \) we may force \( \ell >0 \).  We get
  \begin{equation*}
    Q = \ell
    \begin{pmatrix}
      c^2+1&cs&s^2\\
      cs&s^2&-cs
    \end{pmatrix}
    .
  \end{equation*}
  The last two columns specify the exterior derivative \( d \) on \(
  \um(\g)^*\cong \g^*/\im\chi^* \).  One sees that \( \um(\g)\cong
  \lr_{3,-1} \) as \( B \) acts with eigenvalues \( \pm\ell s \).
\end{proof}

We may describe the isomorphism of~\( \g \) with \( \affR\times\affR \)
explicitly by introducing half-angles.  Writing \( c = \sigma^2-\tau^2 \),
\( s = 2\sigma\tau \), \( \sigma^2+\tau^2 = 1 \), \( \sigma >0 \) and using
the orthogonal transformation \( a' = \sigma a + \tau b \), \( b' = -\tau a
+ \sigma b \), gives the structural equations
\begin{equation*}
  d(Ja') = 2\ell\sigma\, a'Ja',\qquad d(Jb') = -2\ell\tau\, b'Jb'.
\end{equation*}
We have \( \ell,\sigma>0 \) and, replacing \( b' \) by \( -b' \) if
necessary, we may ensure that \( \tau<0 \).  The \SKT moduli space is thus
parameterised by \( \sigma/\tau\in (-1,0) \), \( \ell>0 \) and the
parameter \( t = g(b',Ja')\in (-1,1) \) in the metric.  Up to homotheties
it is connected of dimension~\( 2 \).  The solutions are Kähler precisely
when \( t = 0 \).

\begin{remark}
  \label{rem:c-vs-SKT}
  If one considers the complex structure on \( \affR\times\affR \) with \(
  da = 0 \), \( d(Ja) = aJa \), \( db = 0 \), \( d(Jb) = bJb \) one sees
  that a metric with \( \omega = aJa + bJb + t(aJb+bJa) \) is \SKT (indeed
  Kähler) only if \( t = 0 \).  Thus for a given complex structure the \SKT
  condition depends on the choice of metric.  This is in contrast to the
  study of \SKT structures on six-dimensional
  nilmanifolds~\cite{Fino-PS:SKT}.
\end{remark}

\subsection{Three-dimensional Abelian derived algebra}
\label{sec:3dA}
For \( \g' = \bR^3 \), we have \( \dim W_1 = 1 \), and moreover the
assumption that \( \g' \) is Abelian implies that \( d(Ja),\,db,\,d(Jb) \in
\I(a) \). So it is legitimate to assume that \( V_2 = JV_2 \). The
structural equations are thus
\begin{gather*}
  da = 0,\quad d(Ja) = x_1aJa,\\
  db = y_1aJa+y_2ab+y_3aJb,\quad d(Jb) = u_1aJa-y_3ab+y_2aJb.
\end{gather*}
with coefficients satisfying the equation
\begin{equation*}
  0 = y_2(2y_2+x_1)
\end{equation*}
and non-degeneracy conditions \( x_1\ne0 \), \( y_2^2+y_3^2\ne0 \).  One
may choose \( a,b \) so that \( x_1>0 \), \( y_1\geqslant 0 \) and \( u_1 =
0 \).  The solutions are then Kähler only if \( y_1 \) and \( y_2 \) are
zero.

If \( y_2 = 0 \), then \( y_3\ne0 \) and \( \g\cong
\lr'_{4,\abs{x_1/y_3},0} \).  Thus the \SKT moduli up to homothety has
dimension~\( 1 \), parameter~\( y_3 \), has two connected components
determined by the sign of~\( y_3 \), and contains the Kähler solutions
as~\( y_1 = 0 \).

For \( y_2\ne0 \), we have \( x_1 = -2y_2 \).  There are two cases.  For \(
y_3 = 0 \), we have \( \g \cong \lr_{4,-1/2,-1/2} \) and there is a
one-dimensional connected family of solutions up to homothety.  For \(
y_3\ne0 \), the Lie algebra \( \g \) is \( \lr'_{4,2\lambda,-\lambda} \)
with \( \lambda = \abs{y_2/y_3} \).  Again the moduli is of dimension \( 1
\) up to homothety and has two connected components.

\subsection{Three-dimensional non-Abelian derived algebra}
\label{sec:3dnA}

For \( \g' = \h_3 \), as above we have \( \dim W_1 = 1 \).  Let \( d' \)
denote the exterior derivative on~\( \g' \).  We distinguish between the
complex and real cases \( V_2 = JV_2 \) and \( V_2 \cap JV_2 = \{0\} \).

\subsubsection{Complex case}
\label{sec:3dnAC}

We have \( a\in W_1 = V_1 \), and \( Ja\in V_2 = JV_2 \).  Moreover it is
possible to take \( b\in V_2^{\perp} \) with \( d'b = 0 \).  The condition
\( \g'\cong \h_3 \) then forces \( d'(Jb) \in \Span{bJa} \), giving the
structural equations
\begin{gather*}
  da = 0,\quad
  d(Ja) = x_1aJa,\\
  db = y_1aJa+y_2ab+y_3aJb,\quad d(Jb) = u_1aJa+u_2ab+u_3aJb+v_1bJa,
\end{gather*}
with \( x_1 \), \( y_2^2+y_3^2 \) and \( v_1 \) non-zero.  Adjusting the
choice of \( a \), we may take \( x_1>0 \).  The \SKT equations are now the
vanishing of
\begin{gather*}
  y_2-u_3+v_1,\quad y_3+u_2,\quad y_3v_1,\\
  v_1(x_1+y_2-u_3),\quad (y_2+u_3)(y_2+u_3+x_1).
\end{gather*}
We deduce that \( y_3 = 0 = u_2 \), \( v_1 = x_1 \) and \( u_3 = y_2+x_1
\), leaving the condition \( (2y_2+x_1)(y_2+x_1) = 0 \).

If \( y_2 = -x_1 \), then the structural equations are
\begin{gather*}
  da = 0,\quad
  d(Ja) = x_1aJa,\\
  db = y_1aJa-x_1ab,\quad d(Jb) = u_1aJa+x_1bJa
\end{gather*}
subject only to \( x_1>0 \).  We see that \( \g/{\z(\g')} \) is isomorphic
to \( \lr_{3,-1} \), so \( \g \) itself is isomorphic to \( \ld_4 \).  The
\SKT moduli modulo homotheties is connected and has dimension~\( 2 \).
There are no Kähler solutions.

For \( x_1 = -2y_2 \), we have the structural equations
\begin{gather*}
  da = 0,\quad
  d(Ja) = x_1aJa,\\
  db = y_1aJa-\tfrac12x_1ab,\quad d(Jb) = u_1aJa+\tfrac12x_1aJb+x_1bJa,
\end{gather*}
again with \( x_1>0 \).  The quotient \( \g/{\z(\g')} \) is isomorphic to
\( \lr_{3,-1/2} \), and \( \g \) is thus isomorphic to \( \ld_{4,2} \).
The solutions are Kähler only for \( y_1 = 0 = u_1 \).  Again the \SKT
moduli space up to homotheties is connected of dimension~\( 2 \).
	
\subsubsection{Real case}
\label{sec:3dnAR}

First note that \( \dim W_2 = 3 \), so we may choose \( b \) to be a unit
vector in \( W_2 \cap \Span{a,Ja}^\bot \).  This gives \( t = g(b,Ja) = 0
\).  Now \( d'b = 0 \), where \( d' \) is the differential on \( \g' \), as
above.  As \( \h_3' = \bR \), we have that \( d'(Ja) \) and \( d'(Jb) \)
are linearly dependent, but not both zero.  In fact, if \( d'(Ja) = 0 \),
we may take \( V_2 = \Span{a,Ja} \) and reduce to the complex case
of~\S\ref{sec:3dnAC}, so we assume instead \( d'(Ja)\ne0 \).

Write \( (x_2,x_3,y_2) = m\mathbf p \), \( (w_1,v_1,v_2) = n\mathbf p \)
for some unit vector \( \mathbf p = (p,q,r) \), \( m\ne0 \).  The
structural equations of \( \h_3 \), imply \( b\wedge d'x = 0 \) is zero for
all \( x\in \g' \), giving \( p = 0 \) and \( x_2 = 0 = w_1 \).  Now \(
q^2+r^2 = 1 \) and one may normalise so that \( r\geqslant0 \).  Then
\begin{equation*}
  d'(Ja) = m\,bJc,\quad d'(Jb) = n\,bJc,
\end{equation*}
where
\begin{equation*}
  c = qa+rb.
\end{equation*}
From this one sees \( d'(nJa-mJb) = 0 \) and so \( (nJa-mJb)\wedge d'x = 0
\) is zero too.  We conclude that \( qJa+rJb \) and \( nJa-mJb \) are
parallel and write \( n = kq \), \( m = -kr \), for some \( k\ne0 \).

The structural equations are now
\begin{gather*}
  da = 0,\quad
  d(Ja) = x_1aJa-kqr(aJb+bJa)-kr^2\,bJb,\\
  db = z_1aJa+z_2ab+z_3aJb,\\
  d(Jb) = u_1aJa+u_2ab+u_3aJb+kq^2\,bJa+kqr\,bJb,
\end{gather*}
with \( q^2+r^2 = 1 \), \( r>0 \), the forms \( d(Ja) \), \( db \), \(
d(Jb) \) non-zero, and subject to
\begin{equation}
  \label{eq:h3}
  \begin{gathered}
    u_3 = z_2 + kq,\quad u_2 = -z_3,\quad rz_1 = qz_3,\\
    kq^3-qz_2-ru_1 = 0,\quad 2kq^2+x_1-z_2-u_3 = 0,\\
    q(q(x_1+z_2-u_3)-2ru_1) = 0,\quad (x_1+z_2+u_3)(z_2+u_3+kr^2) = 0.
  \end{gathered}
\end{equation}

Substituting the first three equations in to the remaining four, one sees
that the first equation on the last line follows from the two on the middle
line.  There are thus two cases corresponding to the two factors of the
last equation.

The first case is \( z_2 = -x_1-u_3 \), which reduces to \( x_1 = -kq^2 = -
u_3 \), \( z_2 = 0 \), \( u_1 = kq^3/r \), giving the structural equations
\begin{gather*}
  da = 0,\quad d(Ja) = -k\,cJc,\quad db = z_3r^{-1}\,aJc,\quad d(Jb) =
  -z_3\,ab + kqr^{-1}\,cJc,
\end{gather*}
Now \( \tilde\g^* = \g/\z(\g')^*\cong \Span{a,b,c} \), with \( c' = c/r \),
has structural equations \( \tilde d a = 0 \), \( \tilde d b = z_3 ac' \),
\( \tilde d c' = -z_3 ab \) and so is isomorphic to \( \lr'_{3,0} \).  This
gives \( \g\cong \ld'_{4,0} \).

In this case the solutions are never Kähler.  The \SKT moduli up to
homotheties has dimension~\( 2 \) and is connected.  To see this note that
\( a \) is specified up to sign, which may be fixed by requiring \( k>0 \),
and replacing \( b \) by \( \pm b \), we may then ensure \( z_3>0 \).  If
\( q\ne0 \) this uniquely specifies \( b \), and the remaining parameter is
given by~\( q \).  For \( q = 0 \), we may rotate in the \( b,Jb \) plan,
but this does not change the solution.

The final case is \( z_2= -u_3-kr^2 \).  Here one finds \( x_1 = -k(1+q^2)
\), \( z_2 = -k/2 \), \( u_1 = - kq(2q^2+1)/2r \) giving
\begin{equation}
  \label{eq:h3last}
  \begin{gathered}
    da = 0,\quad d(Ja) = -k (aJa + cJc),\quad
    db = -\tfrac12k\,ab+z_3r^{-1}\,aJc,\\
    d(Jb)= \tfrac12kr^{-1}\, a(qJa-rJb) -z_3\,ab + kqr^{-1}\,cJc.
  \end{gathered}
\end{equation}
This time computing the structural equations for \( \tilde \g = \g/\z(\g')
\) gives \( \tilde d a= 0 \), \( \tilde d b = -\tfrac12kab+z_3ac' \), \(
\tilde d c' = -z_3ab-\tfrac12k ac' \).  If \( z_3\ne0 \), we have \(
\tilde\g \cong \lr'_{3,\lambda} \) with \( \lambda = \abs{k/2z_3} \) giving
\( \g\cong \ld'_{4,\lambda} \).  For \( z_3 = 0 \), we have \(
\tilde\g\cong \lr_{3,1} \) and \( \g \cong \ld_{4,1/2} \).

The solutions are Kähler precisely when \( q = 0 \).  The \SKT moduli up to
homotheties has dimension~\( 1 \) and is connected both for \( \g =
\ld'_{4,\lambda} \) and for \( \g = \ld_{4,1/2} \).

This completes the proof of Theorem~\ref{thm:clsfskt1}.

\section{Consequences and concluding remarks}
\label{sec:consequences}

Let us first emphasise Remark~\ref{rem:c-vs-SKT} that for four-dimensional
solvable groups the \SKT condition depends explicitly on both the metric
and the complex structure, in contrast to the situation~\cite{Fino-PS:SKT}
for six-dimensional nilpotent groups.

\begin{corollary}
  There are four-dimensional solvable complex Lie groups whose family of
  compatible invariant Hermitian metrics contains both \SKT and non-\SKT
  structures.
\end{corollary}

An alternative approach to our classification of invariant \SKT structures
in Theorem~\ref{thm:clsfskt1} would be to start with results for complex
structures on four-dimensional solvable Lie groups
(Ovando~\cite{Ovando:complex-solvable,Ovando:4},
Snow~\cite{Snow:complex-solvable}) and then to impose the \SKT condition.
We have used this approach to cross check our results, but also found that
the lists given in \cite{Ovando:4} for Kähler forms and algebras with
complex structures have some errors and omissions.  Some of these are
corrected in~\cite{Andrada-BDO:four}, but we wish to emphasise that the
proof given in~\S\ref{sec:clsfskt} is independent of those calculations.
In contrast to the compact case we see:

\begin{corollary}
  The four-dimensional solvable Lie algebras~\( \g \) that admit invariant
  complex structures but no compatible invariant \SKT metric are: \( \bR
  \times \lr_{3,1} \), \( \bR \times \lr'_{3,\lambda>0} \), \( \aff_\bC \),
  \( \lr_{4,1} \), \( \lr_{4,\mu,\lambda} \), \textup{(}\( \mu =
  \lambda\ne-\tfrac12 \) or \( \mu\leqslant\lambda=1 \)\textup{)}, \(
  \lr'_{4,\mu,\lambda} \) \( (\lambda\ne0,-\mu/2) \), \( \ld_{4,\lambda} \)
  \( (\lambda\ne\tfrac12,2) \), \( \h_4 \).  \qed
\end{corollary}

\noindent
Here the given constraints on the parameters are in addition to the
defining constraints for the algebras.

On the other hand if \( G \) admits a discrete co-compact subgroup \(
\Gamma \) then \( M = \Gamma\backslash G \) is a compact manifold (a
solvmanifold).  By Gauduchon's Theorem~\cite{Gauduchon:torsion} any complex
structure on~\( M \) admits an \SKT metric (indeed one in any compatible
conformal class).  If \( G \) has an invariant complex structure one may
then construct a compatible invariant \SKT structure on \( G \) via
pull-back from~\( M \) (cf. \cite{Fino-G:SU-Sp}).  A necessary condition
for \( \Gamma \) to exist is that \( G \) be unimodular, which is
equivalent to \( b_4(\g) = 1 \), but in general this is not sufficient.
The correct classification of complex solvmanifolds in dimension four has
recently been provided by Hasegawa \cite{Hasegawa:complex-kaehler}.  In our
notation, one obtains
\begin{inparaenum}
\item tori from \( \g = \bR^4 \),
\item primary Kodaira surfaces from \( \g = \bR\times \h_3 \),
\item hyperelliptic surfaces from \( \g = \bR\times \lr'_{3,0} \),
\item Inoue surfaces of type \( S^0 \) from \( \g =
  \lr_{4,-\frac12,-\frac12} \) and from \( \g = \lr'_{4,2\lambda,-\lambda}
  \),
\item Inoue surfaces of type \( S^\pm \) from \( \g = \ld_4 \) and
\item secondary Kodaira surfaces from \( \g = \ld'_{4,0} \).
\end{inparaenum}
Comparing this list with our classification we conclude:

\begin{corollary}
  Each unimodular solvable four-dimensional Lie group \( G \) with
  invariant \SKT structure admits a compact quotient by a lattice.  \qed
\end{corollary}

Recall that an \HKT structure is given by three complex structures \( I \),
\( J \), \( K = IJ = -JI \) with common Hermitian metric~ such that \(
Id\omega_I = Jd\omega_J = Kd\omega_K \).  If \( (g,I) \) is already \SKT
then \( (g,J) \) and \( (g,K) \) are necessarily \SKT and the \HKT
structure is strong.  However the list of \HKT structures on solvable Lie
groups is known in dimension four from \cite{Barberis:hc4}.

\begin{corollary}
  The only four-dimensional solvable Lie algebra that is strong \HKT is~\(
  \bR^4 \), which is hyperKähler.  The algebra \( \ld_{4,1/2} \) admits
  both \HKT and \SKT structures; these structures are distinct.  The
  remaining \HKT algebras \( \aff_\bC \) and \( \lr_{4,1,1} \) do not admit
  invariant \SKT structures. \qed
\end{corollary}

In the case of \( \ld_{4,1/2} \) one may use~\eqref{eq:h3last} to check
that the \HKT and \SKT metrics are different.

Finally, let us make the following observation which follows from
case-by-case study of the algebras found in our \SKT classification
Theorem~\ref{thm:clsfskt1}.

\begin{corollary}
  Each invariant \SKT structure on a four-dimensional solvable Lie group~\(
  G \) is invariant under a two-dimensional Abelian subgroup \( H \leqslant
  G \). \qed
\end{corollary}

This motivates a future study of \SKT structures on Abelian principal
bundles over Riemann surfaces.

\appendix

\section{SKT structures on compact Lie groups}
\label{sec:compact}

The existence of \SKT structures on compact even-dimensional Lie groups, is
briefly indicated in the introduction to \cite{Fino-PS:SKT}, and attributed
to~\cite{Spindel-STvP:complex}.  However, the result is not explicit in the
latter reference and neither specifies the complex structures.  We
therefore give a proof for reference.

\begin{proposition}
  Any even-dimensional compact Lie group~\( G \) admits a left-invariant
  \SKT structure.  Moreover each left-invariant complex structure on~\( G
  \) admits a compatible invariant \SKT metric.
\end{proposition}

\begin{proof}
  Let \( \lt^\bC \) be a Cartan subalgebra of \( \g^\bC \).
  By~\cite{Samelson:complex}, left-invariant complex structures~\( J \) on
  \( G \) are in one-to-one correspondence with pairs \( (J_{\lt},P) \),
  where \( J_\lt \) is any complex structure on \( \lt \) and \(
  P\subseteq\Delta \) is a system of positive roots: one defines
  \begin{equation}
    \g^{1,0} = \lt^{1,0}\oplus\bigoplus_{\alpha\in P}\g^\bC_\alpha.
  \end{equation}
  Extend the negative of the Killing form on \( [\g,\g] \) to a \( J
  \)-compatible positive definite inner product~\( g \) on~\( \g \).  The
  associated Levi-Civita connection on~\( G \) has \( \LC_XY =
  \tfrac12[X,Y] \), for \( X,Y\in\g \).  Consider now the left-invariant
  connection given by
  \begin{equation}
    \NB_XY = 0,\quad \text{for \( X,Y\in\g \)}.
  \end{equation}
  This connection preserves the metric \( g \) and the complex structure \(
  J \) and has torsion \( T^\NB(X,Y) = -[X,Y] \), so \(
  (T^\NB)^{\flat}(X,Y,Z) = -g([X,Y],Z) \), which is a closed three-form.
  Thus \( (G,g,J) \) is an \SKT manifold.
\end{proof}

\providecommand{\bysame}{\leavevmode\hbox to3em{\hrulefill}\thinspace}
\providecommand{\MR}{\relax\ifhmode\unskip\space\fi MR }
\providecommand{\MRhref}[2]{%
  \href{http://www.ams.org/mathscinet-getitem?mr=#1}{#2}
}
\providecommand{\href}[2]{#2}

\clearpage\begin{small}
  \parindent0pt\parskip\baselineskip

  T.B.Madsen \& A.F.Swann

  Department of Mathematics and Computer Science, University of Southern
  Denmark, Campusvej 55, DK-5230 Odense M, Denmark

  \textit{and}

  CP\textsuperscript3-Origins, Centre of Excellence for Particle Physics
  Phenomenology, University of Southern Denmark, Campusvej 55, DK-5230
  Odense M, Denmark.

  \textit{E-mail}: \url{tbmadsen@imada.sdu.dk}, \url{swann@imada.sdu.dk}
\end{small}

\end{document}